\documentclass{amsart}

\usepackage{amsmath,amsxtra,amsthm,amssymb,amsfonts,graphicx} 

\usepackage[pdfusetitle,bookmarks=true,bookmarksnumbered=false,bookmarksopen=false,breaklinks=false,pdfborder={0 0 0},backref=false,colorlinks=true]{hyperref}

\theoremstyle{plain}

\newtheorem{thm}{Theorem}[section]

\newtheorem{lem}[thm]{Lemma}

\theoremstyle{definition}

\newtheorem{rem}[thm]{Remark}

\def\l{\left}
\def\r{\right}

\def\as{\!\mathrel{\mathop:}=} 

\def\fun{\col\hs\to\exrls}

\def\geo{\col[0,1]\to\hs}

\def\nat{\mathbb{N}}

\def\rls{\mathbb{R}}
\def\exrls{(-\infty,\infty]}

\def\lam{\lambda}
\def\gam{\gamma}
\def\sph{\mathbb{S}}

\def\half{\frac{1}{2}}

\renewcommand{\phi}{\varphi}

\def\id{\mathrm{id}}
\def\hm{d_{\mathrm{H}}}
\def\col{\colon}
\def\ol{\overline}
\def\clco{\operatorname{\ol{co}}}
\def\argmin{\operatornamewithlimits{\arg\min}}
\def\diam{\operatorname{diam}}
\def\dom{\operatorname{dom}}

\def\cldom{\operatorname{\ol{dom}}}

\def\mi{\operatorname{Min}}

\def\pb{\overline{p}}
\def\qb{\overline{q}}
\def\rb{\overline{r}}
\def\xb{\overline{x}}
\def\yb{\overline{y}}
\def\yob{\overline{y_1}}
\def\ytb{\overline{y_2}}
\def\zob{\overline{z_1}}
\def\ztb{\overline{z_2}}
\def\uib{\overline{u_i}}

\def\hs{\mathcal{H}}   

\begin{document}
\title[Lipschitz retractions]{Lipschitz retractions in Hadamard spaces via gradient flow semigroups}
\author[M. Ba\v{c}\'ak]{Miroslav Ba\v{c}\'ak}
\date{\today}
\thanks{}

\address{Max Planck Institute for Mathematics in the Sciences, Inselstr. 22, 04 103 Leipzig, Germany}
\email{bacak@mis.mpg.de}

\author[L. V. Kovalev]{Leonid V. Kovalev}
\address{215 Carnegie, Syracuse University, Syracuse, NY 13244, USA}
\email{lvkovale@syr.edu}
\thanks{L. V. Kovalev was supported by the National Science Foundation grant DMS-1362453.}

\begin{abstract}
Let $X(n),$ for $n\in\nat,$ be the set of all subsets of a metric space $(X,d)$ of cardinality at most~$n.$ The set~$X(n)$ equipped with the Hausdorff metric is called a finite subset space. In this paper we are concerned with the existence of Lipschitz retractions $r\col X(n)\to X(n-1)$ for $n\ge2.$ It is known that such retractions do not exist if~$X$ is the one-dimensional sphere. On the other hand L.~Kovalev has recently established their existence in case~$X$ is a Hilbert space and he also posed a question as to whether or not such Lipschitz retractions exist for $X$ being a Hadamard space. In the present paper we answer this question in the positive.
\end{abstract}

\maketitle
\section{Introduction}
Let $(X,d)$ be a metric space. For each $n\in\nat,$ we denote the set of all subsets of~$X$ with cardinality at most~$n$ by~$X(n).$ The set~$X(n)$ equipped with the Hausdorff metric~$\hm$ is called a \emph{finite subset space.} Unlike Cartesian powers $X^n$ or the space of unordered $n$-tuples $X^n/S_n,$ finite subset spaces admit \emph{canonical} isometric embeddings $\iota\col X(n)\to X(n+1).$

Following \cite{kovalev15,kovalev,mostovoy}, we are interested in Lipschitz retractions $r\col X(n)\to X(n-1).$ L. Kovalev proved their existence for~$X$ being a Euclidean space~\cite{kovalev15} and~$X$ being a Hilbert space~\cite{kovalev}. On the other hand, a~result of J.~Mostovoy~\cite{mostovoy} yields that in general there is no continuous mapping $r\col X(n)\to X(n-1)$ with $r\circ\iota =\id$ if~$X$ is the one-dimensional sphere $\sph^1;$ here~$\id$ stands for the identity operator on~$X(n).$ It is therefore natural to ask whether Lipschitz retractions $r\col X(n)\to X(n-1)$ exist if~$X$ is a nonpositively curved metric space. Indeed, this question appears explicitly in \cite[Question 3.3]{kovalev}. Our main result (Theorem \ref{thm:main}) provides the positive solution to this problem.

The solution for Hilbert spaces from \cite{kovalev} is based on the existence of gradient flow trajectories in a finite dimensional subspace, which is assured by the classical ODE theory. In the present paper, we also define the desired retractions via gradient flows of certain convex functionals on (the $n$-th power of) a Hadamard space and make use of the Lie-Trotter-Kato formula proved recently in \cite{lie,stoj}.

Finally, note that, given a Hadamard space $(\hs,d),$ we obtain Lipschitz retractions $r\col \hs(n)\to\hs(n-1)$ with Lipschitz constant $\max(4n^{\frac32}+1,2n^2+n^\half),$ whereas the result from \cite{kovalev} for $X$ being a Hilbert space has Lipschitz constant $\max(2n-1,n^{\frac32}).$  
\section{Preliminaries}
We first recall some basic facts about Hadamard spaces as well as more recent results which shall be used in our proof. For further details, we refer the reader to~\cite{mybook}.

Let $\l(\hs,d\r)$ be a Hadamard space, that is, a complete metric space with geodesics satisfying
\begin{equation} \label{eq:cat}
 d\l(z,x_t\r)^2\leq (1-t)d\l(z,x_0\r)^2+td\l(z,x_1\r)^2- t(1-t)d\l(x_0,x_1\r)^2,
\end{equation}
for each $z,x_0,x_1\in\hs$ and $t\in[0,1],$ where $x_t\as(1-t)x_0 + tx_1$ is a unique point on the geodesic $\l[x_0,x_1\r]$ such that $d\l(x_0,x_t\r)=t d\l(x_0,x_1\r).$ An equivalent (and more geometric) formulation of inequality~\eqref{eq:cat} is the following relation between a triangle with vertices $p,q,r\in\hs$ and its \emph{comparison triangle} with vertices $\pb,\qb,\rb\in\rls^2,$ where $d(p,q)=\l\|\pb-\qb\r\|,d(r,q)=\l\|\rb-\qb\r\|$ and $d(p,r)=\l\|\pb-\rb\r\|.$ If $x\as(1-t)p + tq$ and $y\as(1-s)p + sr$ for some $s,t\in[0,1],$ and we denote their comparison points by $\xb\as(1-t)\pb + t\qb$ and $\yb\as(1-s)\pb + s\rb,$ respectively, inequality~\eqref{eq:cat} implies that
\begin{equation} \label{eq:comparison}
 d\l(x,y\r)\le \l\|\xb-\yb\r\|.
\end{equation}
Here the symbol $\|\cdot\|$ stands for the Euclidean norm on $\rls^2.$

Given two geodesics $\l[x_0,x_1\r]$ and $\l[y_0,y_1\r],$ we have
\begin{equation}  \label{eq:jointconv}
d\l(x_t,y_t\r)\leq (1-t) d\l(x_0,y_0\r)+ td\l(x_1,y_1\r),
\end{equation}
for each $t\in[0,1];$ see \cite[(1.2.4)]{mybook}.

Given a function $f\fun,$ denote its \emph{domain} by $\dom f\as\l\{x\in\hs\col f(x)<\infty\r\}$ and the set of its minimizers by $\mi f\as\l\{x\in\hs\col f(x)=\inf f \r\}.$ We say that a function $f\fun$ is \emph{convex} if for each geodesic $\gamma\geo,$ the function $f\circ\gamma$ is convex. Given a convex lower semicontinuous (lsc, for short) function, its \emph{resolvent} (with parameter $\lam>0$) is given by 
\begin{equation} \label{eq:defres}
J_\lam x\as\argmin_{y\in \hs} \l[f(y)+\frac1{2\lam}d(x,y)^2\r],\qquad x\in \hs,
\end{equation}
and it satisfies the following important inequality
\begin{equation} \label{eq:onestepestim}
  f\l(J_\lam x\r) + \frac1{2\lam} d\l(x,J_\lam x\r)^2 + \frac1{2\lam} d\l(J_\lam x,y\r)^2  \leq f(y)+\frac1{2\lam} d(x,y)^2,
\end{equation}
for every $x,y\in\hs.$ Given $x\in\cldom f,$ the \emph{gradient flow semigroup} associated to~$f$ is defined by
\begin{equation} \label{eq:semigrconstr}
 S_t x\as\lim_{k\to\infty} \l(J_{\frac{t}k}\r)^k x,\qquad x\in\cldom f,
\end{equation}
for every $ t\in[0,\infty).$ Like in Hilbert spaces, the semigroup is comprised of nonexpansive operators, that is
\begin{equation}\label{eq:nonexp}
d\l(S_t x, S_t y\r)\leq d(x,y),  
\end{equation}
for each $t\in[0,\infty)$ and $x,y\in\cldom f.$ The above mentioned theory of gradient flows in Hadamard spaces was first studied by J.~Jost~\cite{jost-ch} and U.~Mayer~\cite{mayer}. A~more recent result from~\cite{ppa} established the asymptotic behavior of a gradient flow. It relies upon the notion of weak convergence in Hadamard spaces, which was introduced by J.~Jost in~\cite{jost94}. Let us recall that a bounded sequence $\l(x_k\r)\subset\hs$ converges \emph{weakly} to a point $x\in\hs$ provided $\lim_{k\to\infty}d\l(P_\gam\l(x_k\r),x\r)=0$ for each geodesic $\gam\geo$ with $x\in\gam.$ Here $P_\gam$ stands for the metric projection onto (the image of) $\gam.$ Now we are ready to state the theorem on asymptotic behavior of a gradient flow. 
\begin{thm} \label{thm:asym}
 Let $f\fun$ be a convex lsc function which attains its infimum on $\hs.$ Then, given $x\in\cldom f,$ the associated gradient flow semigroup weakly converges to a point $x^*\in\mi f.$ 
\end{thm}
\begin{proof} See \cite{ppa} or \cite[Thm 5.1.16]{mybook}.
\end{proof}
The function values then converge to the infimum of~$f.$
\begin{thm} \label{thm:asym-f-value}
 Let $f\fun$ be a convex lsc function which attains its infimum on $\hs.$ Then, given $x\in\cldom f,$ we have $f\l(S_tx\r)\to\inf f$ as $t\to\infty.$ 
\end{thm}
\begin{proof} This can be seen from the proof of Theorem~\ref{thm:asym} in \cite{ppa}, or for instance like in \cite[Prop. 5.1.12]{mybook}.
\end{proof}

The proof of our main theorem uses gradient flows of convex functions to define a desired Lipschitz retraction. These convex functions have a special form, namely they are given as a finite sum of some elementary convex functions, and their gradient flow can be approximated by the Lie-Trotter-Kato formula. We will now state the necessary facts precisely. Let $N\in\nat$ and consider a function $f\fun$ of the form 
\begin{equation} \label{eq:fsum}
f\as\sum_{j=1}^N f_j,
\end{equation}
where $f_j\fun$ are convex lsc functions for every $j=1,\dots,N.$ Let us denote the resolvent of the function~$f_j$ by~$J_\lam^{[j]}$ and the gradient flow semigroup of~$f$ by~$S_t.$
\begin{thm}[Lie-Trotter-Kato formula]
Let $f\fun$ be of the form~\eqref{eq:fsum}. Then we have
\begin{equation} \label{eq:lie}
S_tx = \lim_{k\to\infty} \l(J_{\frac{t}{k}}^{[N]}\circ\dots\circ J_{\frac{t}{k}}^{[1]}\r)^{k}x , 
\end{equation}
for every $t\in[0,\infty)$ and $x\in\cldom f.$
\end{thm}
\begin{proof}
 The original proof appeared in~\cite{stoj}. For a simplified proof, see~\cite{lie}.
\end{proof}

In fact, the gradient flow we are going to use in the proof of Theorem~\ref{thm:main} is not on~$\hs,$ but on its $n$-th power. The $n$-th power of~$\hs,$ denoted by~$\hs^n,$ is equipped with the metric
\begin{equation*}d(x,y)\as\biggl(\sum_{j=1}^n d\l(x_j,y_j\r)^2 \biggr)^\half,\qquad x,y\in\hs^n,\end{equation*}
and is then also a Hadamard space. Note that we use the same symbol $d$ for the original metric on $\hs$ as well as for the metric on $\hs^n.$ It is always clear from the arguments which one is meant. Given a point $x=\l(x_1,\dots,x_n\r)\in\hs^n,$ we shall denote $\{x\}\as\l\{x_1,\dots,x_n\r\},$ a subset of $\hs.$ However, given $x,y\in\hs^n,$ we shall denote the Hausdorff distace between $\{x\}$ and $\{y\}$ by $\hm(x,y)$ instead of $\hm\l(\{x\},\{y\}\r).$

\section{The existence of Lipschitz retractions}
The desired Lipschitz retractions $r\col\hs(n)\to\hs(n-1)$ will be defined via a gradient flow of a convex functional on~$\hs^n.$ Specifically, we define this functional as
\begin{equation} \label{eq:f}
 F(x)\as\sum_{1\le i<j\le n} d\l(x_i,x_j\r),\qquad x=\l(x_1,\dots,x_n\r)\in\hs^n.
\end{equation}
and show that it is indeed convex and Lipschitz.
\begin{lem}
 The function $F\col\hs^n\to\rls$ is convex and $n^{\frac32}$-Lipschitz. 
\end{lem}
\begin{proof}
 Convexity follows from \eqref{eq:jointconv}. For the Lipschitz property, we estimate
 \begin{equation*}
  \l|F(x)-F(y)\r| \le \sum_{1\le i<j\le n} \l|d\l(x_i,x_j\r)-d\l(y_i,y_j\r) \r| 
  \le \sum_{1\le i<j\le n} d\l(x_i,y_i\r)+d\l(x_j,y_j\r) \le n^{\frac32} d(x,y),
 \end{equation*}
where we twice used the triangle inequality and then the Cauchy-Scharz inequality.
\end{proof}

We are now ready to prove the main theorem.
\begin{thm}\label{thm:main}
 Let $(\hs,d)$ be a Hadamard space. Then for each integer $n\ge2$ there exists a Lipschitz retraction $r\col\hs(n)\to\hs(n-1)$ with Lipschitz constant $\max(4n^{\frac32}+1,2n^2+n^\half).$
\end{thm}
\begin{proof} 
We divide the proof into several steps.

\textbf{Step 1.} Let $J_\lam$ and $S_t$ be the resolvent and gradient flow semigroup, respectively, associated with the function $F$ from \eqref{eq:f}. Let us denote
\begin{equation*}
 D\as\l\{x=\l(x_1,\dots,x_n\r)\in\hs^n\col x_i=x_j \textnormal{ for some } 1\le i<j\le n \r\}.
\end{equation*}
Given $x\in\hs^n,$ we define
\begin{equation}
 \delta(x)\as \min_{1\le i<j\le n} d\l(x_i,x_j\r).
\end{equation}
and $T(x)\as\inf\l\{t>0\col S_tx\in D\r\}.$ We will first show that
\begin{equation}\label{eq:tbound}
T(x)\le\half\delta(x). 
\end{equation}
In order to be able to apply the Lie-Trotter-Kato formula \eqref{eq:lie}, denote by $J_\lam^{[i,j]}$ the resolvent associated with the function
\begin{equation*} 
 x\mapsto d\l(x_i,x_j\r),\qquad x=\l(x_1,\dots,x_n\r)\in\hs^n,
\end{equation*}
where $1\le i<j\le n.$ Then \eqref{eq:lie} reads
\begin{equation} \label{eq:lief}
S_tx = \lim_{k\to\infty} \l(R_{\frac{t}{k}}\r)^k x,\qquad x\in\cldom f,
\end{equation}
where
\begin{equation*}
R_{\frac{t}{k}}\as J_{\frac{t}{k}}^{[n-1,n]}\circ\dots\circ  J_{\frac{t}{k}}^{[1,6]}\circ J_{\frac{t}{k}}^{[4,5]}\circ  J_{\frac{t}{k}}^{[3,5]}\circ  J_{\frac{t}{k}}^{[2,5]}\circ  J_{\frac{t}{k}}^{[1,5]}\circ J_{\frac{t}{k}}^{[3,4]}\circ J_{\frac{t}{k}}^{[2,4]}\circ J_{\frac{t}{k}}^{[1,4]}\circ J_{\frac{t}{k}}^{[2,3]}\circ J_{\frac{t}{k}}^{[1,3]}\circ J_{\frac{t}{k}}^{[1,2]}.
\end{equation*}

Next we claim that for $\lam>0$ and $i=3,\dots,n,$ the following holds:
\begin{enumerate}
 \item if $d\l(y_1,y_2\r)\ge\lam,$ then $ d\l(z_1,z_2\r)\le d\l(y_1,y_2\r),$ \label{i:i}
 \item if $d\l(y_1,y_2\r)<\lam,$ then $ d\l(z_1,z_2\r)\le d\l(y_1,y_2\r)+\lam,$ \label{i:ii}
\end{enumerate}
for every $y=\l(y_1,\dots,y_n\r)\in\hs^n$ and $z\as J_\lam^{[2,i]}J_\lam^{[1,i]}y.$ We will now show both~\eqref{i:i} and~\eqref{i:ii} by using comparison triangles. To this end denote $u\as J_\lam^{[1,i]}y$ and consider the triangle with vertices $y_1,y_2,u_i$ in $\hs$ along with its comparison triangle with vertices $\yob,\ytb,\uib\in\rls^2.$ Then denote the comparison points of $z_1$ and $z_2$ by $\zob$ and $\ztb,$ respectively. Next observe that~\eqref{i:i} and~\eqref{i:ii} hold true if we replace all the points involved by their comparison points (and consider the Euclidean distance in $\rls^2$ of course). This can be seen by elementary geometry arguments in $\rls^2.$ Finally, applying~\eqref{eq:comparison} gives~\eqref{i:i} and~\eqref{i:ii}.

Choose $x\in\hs^n$ and $k\in\nat.$ Denote $\lam\as\frac{\delta(x)}{2k}.$ Without loss of generality one may assume $d\l(x_1,x_2\r) = \delta(x).$ Define now
\begin{equation*}
 x^{l,[i,j]}\as J_\lam^{[i,j]}\circ\cdots\circ J_\lam^{[2,3]} \circ J_\lam^{[1,3]} \circ J_\lam^{[1,2]} \circ \l(R_{\frac{t}{k}}\r)^{l-1} x,
\end{equation*}
for each $l=1,\dots,k$ and $1\le i<j\le n,$ and observe that~\eqref{i:i} and~\eqref{i:ii} imply
\begin{equation} \label{eq:laststep}
 d\l(x_1^{k,[1,n]},x_2^{k,[2,n]}\r)\le\lam = \frac{\delta(x)}{2k},
\end{equation}
where the subscript indices denote the coordinates in $\hs^n.$ Indeed, each application of~$J_\lam^{[1,2]}$ shortens the distance between the first two coordinates by additive constant~$2\lam$ while the application of any other resolvent than~$J_\lam^{[1,2]}$ does not expand it, or expands it by additive constant~$\lam$ at most --- as we know from~\eqref{i:i} and~\eqref{i:ii}. More precisely, we have
\begin{align*}
 d\l(x_1,x_2\r) &= \delta(x), \\
 d\l(x_1^{1,[1,2]},x_2^{1,[1,2]}\r) &=\max\l(0,\delta(x)-\frac{\delta(x)}{k}\r) ,\\
 d\l(x_1^{1,[n-1,n]},x_2^{1,[n-1,n]}\r) &\le \max\l(\frac{\delta(x)}{k},\delta(x)-\frac{\delta(x)}{k}\r), \\
 d\l(x_1^{2,[1,2]},x_2^{2,[1,2]}\r) &\le\max\l(0,\delta(x)-2\frac{\delta(x)}{k}\r), \\
 d\l(x_1^{2,[n-1,n]},x_2^{2,[n-1,n]}\r) &\le \max\l(\frac{\delta(x)}{k},\delta(x)-2\frac{\delta(x)}{k}\r), \\
 & \vdots \\
 d\l(x_1^{k-1,[n-1,n]},x_2^{k-1,[n-1,n]}\r) &\le \max\l(\frac{\delta(x)}{k},\delta(x)-(k-1)\frac{\delta(x)}{k}\r)=\frac{\delta(x)}{k}, \\
 d\l(x_1^{k,[1,2]},x_2^{k,[1,2]}\r) &  =0.
\end{align*}
and hence~\eqref{eq:laststep} holds true.

Passing to the limit $k\to\infty$ in~\eqref{eq:laststep} and recalling \eqref{eq:lief} then give
\begin{equation*}
 d\l(\l(S_{\half\delta(x)}x\r)_1,\l(S_{\half\delta(x)}x\r)_2\r)=0,
\end{equation*}
or, in other words, we have just proved~\eqref{eq:tbound}.

\textbf{Step 2.}
Let $x,y\in\hs^n.$ By \eqref{eq:onestepestim} we have
\begin{equation*}
 \frac1{2\lam} d\l(J_\lam x,y\r)^2  \leq F(y)-F\l(J_\lam x\r) +\frac1{2\lam} d(x,y)^2.
\end{equation*}
Consider now $t\in\l[0,T(x)\r].$ Fix $k\in\nat$ and employ the above inequality $k$-times to obtain
\begin{align*}
 d\l(J_{\frac{t}{k}}x,y\r)^2 & \leq \frac{2t}{k}\l[F(y)-F\l(J_{\frac{t}{k}}x\r)\r] +d(x,y)^2, \\
d\l(\l(J_{\frac{t}{k}}\r)^2 x,y\r)^2 & \leq \frac{2t}{k}\l[F(y)-F\l(\l(J_{\frac{t}{k}}\r)^2 x\r)\r] +d\l(J_{\frac{t}{k}}x,y\r)^2, \\
& \vdots \\
d\l(\l(J_{\frac{t}{k}}\r)^k x,y\r)^2 & \leq \frac{2t}{k}\l[F(y)-F\l(\l(J_{\frac{t}{k}}\r)^k x\r)\r] +d\l(\l(J_{\frac{t}{k}}\r)^{k-1}x,y\r)^2 .
\end{align*}
Summing up these inequalities, dividing by $t^2$ and putting $x\as y$ gives
\begin{equation*} \frac{d\l(\l(J_{\frac{t}{k}}\r)^k x,x\r)^2}{t^2} \leq 2\frac{F\l(x\r)-F\l(\l(J_{\frac{t}{k}}\r)^k x\r)}{t} \leq 2n^{\frac32}\frac{d\l(x,\l(J_{\frac{t}{k}}\r)^k x\r)}{t},\end{equation*}
and after taking $\limsup_{k\to\infty}$ we obtain
\begin{equation*} 
\frac{d\l(S_t x,x\r)}{t} \leq 2n^{\frac32}.
\end{equation*}
Hence, by virtue of~\eqref{eq:tbound},
\begin{equation} \label{eq:spread}
\hm\l(S_t x,x\r) \le d\l(S_t x,x\r) \le 2tn^{\frac32} \le \delta(x) n^{\frac32}.
\end{equation}

For future reference we also record that the nonexpansiveness of the gradient flow semigroup~\eqref{eq:nonexp} implies
\begin{equation} \label{eq:flownonexp}
 \hm\l(S_t x,S_t y\r) \le d\l(S_t x, S_t y\r) \le d(x,y) \le n^{\half} \max_{1\le j\le n}d\l(x_j,y_j\r).
\end{equation}

\textbf{Step 3.} Given $x\in\hs(n),$ we number its elements $\l\{x_1,\dots,x_n\r\}$ and consider $x'\as\l(x_1,\dots,x_n\r)\in\hs^n.$ We may assume that $d\l(x_1,x_2\r)=\delta\l(x'\r).$ Then we define $r(x)\as \l\{S_{T\l(x'\r)} x'\r\}.$ However, we will write $x$ instead of $x'$ in the sequel. Let us now show that $r\col\hs(n)\to\hs(n-1)$ is a Lipschitz retraction. First of all, observe that $r$ is the identity on (the canonical embedding of) $\hs(n-1).$ To prove the Lipschitz property, choose $x,y\in\hs(n)$ and examine the following two alternatives. If $\delta(x)+\delta(y)\le4\hm(x,y),$ then 
\begin{equation*}
 \hm\l(r(x),r(y)\r)\le \hm\l(r(x),x\r)+\hm\l(x,y\r)+\hm\l(y,r(y)\r)\le n^{\frac32} \delta(x)+\hm(x,y)+n^{\frac32}\delta(y)\le \l(4n^{\frac32}+1\r)\hm(x,y)
\end{equation*}
where we used \eqref{eq:spread} to obtain the second inequality.

If, on the other hand, $\delta(x)+\delta(y)>4\hm(x,y),$ then we may assume $\delta(x)>2\hm(x,y)$ without loss of generality. The fact $\delta(x)>2\hm(x,y)$ then implies that we can renumber the points $\l\{y_1,\dots,y_n \r\}$ in such a way that
\begin{equation} \label{eq:lastestim}
 d\l(x_j,y_j\r)\le \hm(x,y)
\end{equation}
for each $j=1,\dots,n.$ In the remainder of the proof, we will use \eqref{eq:lastestim} only, without referring to $\delta(x)>2\hm(x,y).$ We can hence without loss of generality assume $T(x)\le T(y)$ on account that the roles of $x$ and $y$ in \eqref{eq:lastestim} are interchangeable. Recall that $r(x)=S_{T(x)}x$ and put $z\as S_{T(x)}y.$ Inequality~\eqref{eq:flownonexp} implies that $\hm\l(r(x),z\r)\le n^{\half} \hm(x,y).$ Consequently,
\begin{equation*}
 \delta(z)=\delta(z)-\delta\l(r(x)\r) \le 2\hm\l(z,r(x)\r)\le 2n^{\half} \hm(x,y).
\end{equation*}
By \eqref{eq:spread} we have
\begin{equation*}
 \hm\l(z,r(z) \r)\le n^{\frac32}\delta(z)\le 2n^2 \hm(x,y).
\end{equation*}
Finally, one arrives at
\begin{equation*}
 \hm\l(r(x),r(y) \r) \le \hm\l(r(x),z \r)+\hm\l(z,r(z) \r)+\hm\l(r(z),r(y) \r)\le n^\half \hm(x,y)+2n^2 \hm(x,y)+0,
\end{equation*}
where the zero on the right hand side is due to the semigroup property of the gradient flow. The proof is complete.
\end{proof}
\begin{rem}[Asymptotic behavior of the flow]
Given $x\in\hs^n,$ denote $\Delta(x)\as \max_{1\le i<j\le n} d\l(x_i,x_j\r).$ Using the same arguments as in the previous proof, we can show that for $\tau\as \half\Delta(x),$ one obtains $S_\tau x\in\mi F.$ Alternatively, we can obtain the asymptotic behavior of the flow as follows. By Theorem \ref{thm:asym}, given $x\in\hs^n,$ the flow $S_t x$ weakly converges to a point $x^*=\l(x_1^*,\dots,x_n^*\r)\in\mi F;$ and obviously $x_1^*=\cdots=x_n^*.$ Next we show that this convergence is in fact strong. To this end, we first observe that
\begin{equation*}
 x_1^*\in\bigcap_{t\in[0,\infty)} \clco \l\{ S_t x \r\}.
\end{equation*}
Indeed, by virtue of \eqref{eq:semigrconstr} and the semigroup property it is sufficient to show that $\l\{J_\lam x\r\} \subset \clco\{x\}.$ This inclusion however follows directly by a projection argument. Now use Theorem \ref{thm:asym-f-value} to conclude $F\l(S_t x\r)\to0$ and therefore $\diam\clco\l\{ S_t x \r\}\to0.$ We hence have $S_tx\to x^*.$
\end{rem}
\begin{rem}[Open questions]
 We end the paper by posing a few questions, many of which have appeared already in~\cite{kovalev}. The Lipschitz constant of $r\col \hs(n)\to\hs(n-1)$ guaranteed by Theorem~\ref{thm:main} is $\max(4n^{\frac32}+1,2n^2+n^\half).$ Can one improve upon this constant? Can one show that, for every $n\in\nat$ with $n\ge2,$ there exist Lipschitz retractions $r\col \hs(n)\to\hs(n-1)$ with Lipschitz constants independent of~$n$? Can one extend Theorem \ref{thm:main} into spaces of nonpositive curvature in the sense of Busemann? In particular, does an analog of Theorem \ref{thm:main} hold in strictly convex or uniformly convex Banach spaces? Can one extend Theorem \ref{thm:main} into $p$-uniformly convex spaces? Recall that a geodesic metric space $(X,d)$ is called $p$-uniformly convex (for $p\ge2$) if there exists $K>0$ such that
 \begin{equation*} 
 d\l(z,x_t\r)^p\leq (1-t)d\l(z,x_0\r)^p+td\l(z,x_1\r)^p-Kt(1-t)d\l(x_0,x_1\r)^p,
\end{equation*}
for each $z,x_0,x_1\in\hs$ and $t\in[0,1],$ where $x_t\as(1-t)x_0 + tx_1.$ This definition was introduced in \cite[Definition 3.2]{naor-silberman} as a generalization of $p$-uniform convexity in Banach spaces.
\end{rem}


\bibliographystyle{siam}
\bibliography{retracts}

\end{document}